\documentclass[12pt]{article}

\usepackage[margin=1in]{geometry}

\usepackage{amsmath,amsthm,amssymb,graphicx}

\newtheorem{theorem}{Theorem}

\newtheorem{proposition}[theorem]{Proposition}

\newtheorem{corollary}[theorem]{Corollary}
\theoremstyle{definition}
\newtheorem{definition}[theorem]{Definition}

\theoremstyle{remark}

\usepackage[
backend=biber,
style=alphabetic,
]{biblatex}

\renewbibmacro{in:}{}
\DeclareFieldFormat*{title}{\mkbibemph{#1}}
\DeclareFieldFormat{journaltitle}{#1}
\DeclareFieldFormat{pages}{#1}

\usepackage{hyperref}
\hypersetup{
	colorlinks=true,
	linkcolor=blue,
	filecolor=magenta,      
	urlcolor=cyan,
	citecolor=black,
}

\begin{document}
	
\title{Homological Polynomial Coefficients and the Twist Number of Alternating Surface Links}

\author{David A. Will}

\date{}

\maketitle

\begin{abstract}
	For $D$ a reduced alternating surface link diagram, we bound the twist number of $D$ in terms of the coefficients of a polynomial invariant. To this end, we introduce a generalization of the homological Kauffman bracket defined by Krushkal. Combined with work of Futer, Kalfagianni, and Purcell, this yields a bound for the hyperbolic volume of a class of alternating surface links in terms of these coefficients.
\end{abstract}
	
\section{Introduction}

Since its introduction in \cite{jones}, the Jones polynomial has been of key interest in the ongoing effort to find explicit relations between quantum invariants of a link and the geometry or topology of the link complement. It is in this spirit that Dasbach and Lin proved a ``volume-ish theorem" for alternating, prime, hyperbolic knots in \cite{volume-ish} which presents upper and lower bounds for the volume of the knot complement in terms of certain coefficients of the Jones polynomial of the knot. This result was obtained in two steps: first, the authors bound the coefficients in terms of the twist number of the link diagram, and second, work of Lackenby, Agol, and Thurston in \cite{lackenby} bounds the volume in terms of the twist number. Futer, Kalfagianni, and Purcell extended both steps to adequate links in \cite{dehn_filling}. \\

The latter of these steps was recently extended to alternating links in higher-genus surfaces by Kalfagianni and Purcell in \cite{links_surfaces}. It is then natural to consider the first step in the same setting. The difficulty, however, is that the classical Jones polynomial does not capture sufficient information about the embedding of links in surfaces. In this paper, we opt for a three-variable generalization of the homological polynomial defined by Krushkal in \cite{krushkal}.\\

Our polynomial, denoted $\langle D \rangle_\Sigma$, is a Laurent polynomial in $\mathbb{Z}[A^{\pm 1},Z,W]$. Here, $A$ is the usual Kauffman polynomial variable, while $Z$, $W$ record homological info for $D$ on the surface. We will be interested in coefficients of terms having certain fixed degrees in the $A$ variable with minimal degrees in the $Z$ and $W$ variables. The coefficients for the second largest and second smallest degree terms in $A$ are denoted $\alpha_{(1)}'$ and $\beta_{(1)}'$, while those for the third-largest and third-smallest are denoted $\alpha''_{(0)}$ and $\beta''_{(0)}$, respectively. The main purpose of the paper is to establish the following bound for these coefficients in terms of the twist number by generalizing the method in \cite{dehn_filling}. The precise definitions of the polynomial and coefficients are given in Sections \ref{def}, \ref{comp}.\\

\begin{theorem}\label{main theorem}
	Let $\Sigma$ be a closed, orientable surface and let $D$ be a reduced alternating link diagram on $\Sigma$ such that every twist region of $D$ has at least three crossings. Let $\star = \alpha_{(1)}' + \beta_{(1)}' - \alpha''_{(0)} - \beta''_{(0)} + 2$. Then,
	\begin{equation}\label{bounds}
	\frac{1}{3}tw(D) + 1 - g(\Sigma) \leq \star \leq 2 tw(D).
	\end{equation}
\end{theorem}

Combined with the result in \cite{links_surfaces}, this yields the following volume bound.\\

\begin{corollary}\label{vol bound}
	Let $\Sigma$ be a closed orientable surface of genus at least one, and let $L$ be a link that admits a twist-reduced weakly generalized cellularly embedded alternating projection $D$ onto $\Sigma \times \{0\}$ in $Y=\Sigma \times [-1,1]$. Then the interior of $Y\setminus L$ admits a hyperbolic structure. If $\Sigma$ is a torus, then we have
	\begin{equation}
	\frac{v_{\text{oct}}}{4}\star\leq vol(Y\setminus L) < 30v_{\text{tet}}\star,
	\end{equation}
	where $v_{\text{tet}}$ is the volume of a regular ideal tetrahedron, and $v_{\text{oct}}$ is the volume of a regular ideal octahedron.\\
	
	If $\Sigma$ has genus at least two, then we have
	\begin{equation}
	\frac{v_{\text{oct}}}{4}\cdot (\star-6\chi(\Sigma)) \leq vol(Y\setminus L) < 18v_{\text{oct}}\cdot (\star + g(\Sigma) -1 ).
	\end{equation}
\end{corollary}

	We should note that while this paper was in preparation, related results were independently obtained for different choices of polynomials. In \cite{champanerkar2020volumish} Champanerkar and Kofman use the homological Kauffman bracket (see Definition \ref{hom bracket h}), while in \cite{br2020guts} Bavier and Kalfagianni define a polynomial $\langle D \rangle_0$ which agrees with the classical Kauffman bracket, but is defined over states consisting only of contractible loops. In contrast to Theorem \ref{main theorem}, both of these results contain strict equalities between coefficients and the twist number, rather than inequalities. There is, however, an associated cost. In \cite{champanerkar2020volumish} the authors must use a homological version of the twist number, and in \cite{br2020guts} the notion of ``reduced" is stronger which eliminates many of the terms making up $\star$ in \eqref{quant}. Although we achieve only inequalities, our approach has the advantage of recovering the classical twist number without requiring these further conditions.

\subsection{Organization of the paper}
	In Section \ref{def} we give some basic definitions and construct the more general homological polynomial referred to in the introduction. In Section \ref{comp} we define the state graphs $\mathbb{G}_A$ and $\mathbb{G}_B$ and use them to compute certain coefficients of the aforementioned polynomial. Lastly, Section \ref{proofs} contains proofs of Theorem \ref{main theorem} and Corollary \ref{vol bound} obtained by finding bounds on these coefficients in terms of the twist number of a diagram.
	
\subsection{Acknowledgments}
	The author would like to thank Slava Krushkal for many helpful discussions and Ilya Kofman for useful feedback.\\
	
\begin{figure}[h!]
	\centering
	\includegraphics{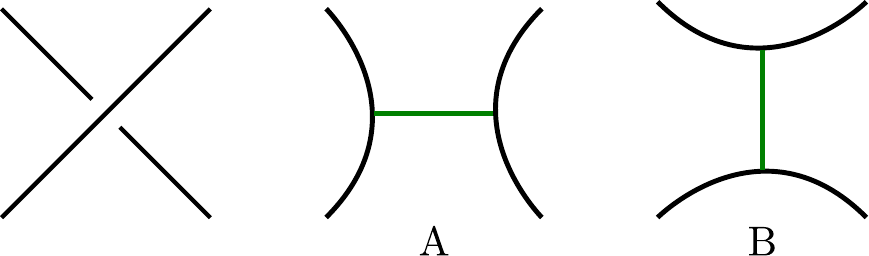}
	\caption{A crossing, an $A$-resolution, and a $B$-resolution.}
\end{figure}

\section{Definitions}\label{def}
Throughout this paper, we will let $\Sigma$ be a closed, orientable surface. We will implicitly view $\Sigma $ as $\Sigma \times \{0\} \subset \Sigma \times [-1,1]$, but since this paper deals primarily with link diagrams, these results will hold for $\Sigma$ in any compact, orientable, irreducible 3-manifold. For a link $L\subset \Sigma \times [-1,1]$, a link diagram $D$ for $L$ on $\Sigma$ is the image of $L$ under the projection, which we may view as a 4-valent graph $\Gamma$ on $\Sigma$. The vertices of $\Gamma$ correspond to double points of the projection, and these are then equipped with crossing information. There are two possible ways to resolve each crossing. Resolving a crossing means replacing $D$ with a new diagram having one fewer crossing which differs from $D$ only in a neighborhood of that crossing. We call these the $A$-resolution and the $B$-resolution, referring to Figure 1. As we resolve crossings, we may draw an arc connecting the strands of the new diagram. These are commonly called ``surgery arcs," as performing surgery along such an arc provides a convenient way to switch between the $A$ and $B$ resolutions of a crossing.\\

If all crossings are resolved, then all that remains of $D$ is a collection $S$ of disjoint simple closed curves on $\Sigma$. We call $S$  a \textit{state}, and denote the set of all states by $\mathcal{S}$. Note that $\lvert \mathcal{S} \rvert = 2^{c(D)}$ where $c(D)$ is the number of crossings of $D$.\\

We will be considering only link diagrams which are alternating and reduced. A link diagram on a surface is \textit{reduced} if it is cellularly embedded and if it contains no nugatory crossings. By cellularly embedded, we mean that the complementary regions of the graph $\Gamma$ are disks. And a crossing is \textit{nugatory} if there exists a separating curve in $\Sigma$ that intersects $D$ only at that crossing as in Figure 2.\\

The polynomial we will be studying is a generalization of the homological Kauffman bracket, defined by Krushkal in \cite{krushkal}. Let $i:S\to\Sigma$ be the inclusion map. ``Homological" refers to the use of the induced map $i_*: H_1(S) \to H_1(\Sigma)$ (with $\mathbb{Z}$ coefficients) which provides additional information about the embedding of each state.\\

\begin{definition}\label{hom bracket h} (\cite{krushkal})
	Let $\alpha(S)$ and $\beta(S)$ denote the number of $A$-resolutions and $B$-resolutions used to obtain the state $S$, respectively. The \textit{homological Kauffman bracket} is a Laurent polynomial $\langle D \rangle_\Sigma^H \in \mathbb{Z}[A^{\pm 1},Z]$ defined by the state-sum
	\begin{equation}
		\langle D \rangle_\Sigma^H = \sum\limits_{S\in \mathcal{S}} A^{\alpha(S)-\beta(S)}(-A^2-A^{-2})^{k(S)} Z^{r(S)},
	\end{equation}
	where $k(S) = \dim(\ker\{i_*: H_1(S) \to H_1(\Sigma)
	\})$ and $r(S) = \dim(\text{im}\{i_*: H_1(S) \to H_1(\Sigma)
	\})$. $k(S)$ is called the \textit{kernel} of the state, while $r(S)$ is called the \textit{homological rank} of the state. We write
	\begin{equation}
		\langle D \mid S \rangle_\Sigma^H = A^{\alpha(S)-\beta(S)}(-A^2-A^{-2})^{k(S)} Z^{r(S)}
	\end{equation}
	for the contribution of the state $S$ to $\langle D \rangle_\Sigma^H$.
\end{definition}

\begin{figure}
	\centering
	\includegraphics[width=3in]{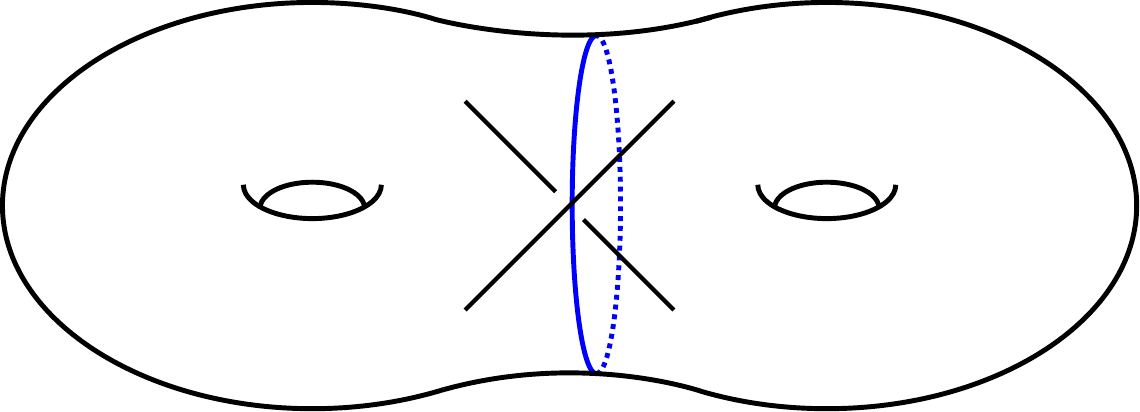}
	\caption{A nugatory crossing.}
\end{figure}

In Lemmas 2.1-2.3 of \cite{b-k}, Boden and Karimi give a proof for the invariance of $\langle D \rangle_\Sigma^H$ under the second and third Reidemeister moves. To obtain invariance under the first Reidemeister move, a homological version of the Jones polynomial may be obtained by substituting $A = t^{-\frac{1}{4}}$ and by setting
\begin{equation}\label{hom jones}
	J_\Sigma^H(t,Z) = (-t^{-\frac{3}{4}})^{w(D)}\langle D \rangle_\Sigma^H,
\end{equation}
where $w(D)$ is the writhe of $D$. Thus, we obtain a link invariant.\\

In \cite{krushkal}, a homological version of the Tutte polynomial is defined for surface graphs, and then the homological Kauffman bracket is a obtained as a specialization of the graph polynomial. More recently, in \cite{krushkal-fendley}, Fendley and Krushkal define a more general version of this graph polynomial for graphs on the torus. More specifically, for a graph with homological rank equal to one, a third variable $W$ is introduced which counts the number of components of the graph which are homologically essential. Inspired by this generalization, we present a version of this polynomial for alternating, cellularly embedded link diagrams, and extend it to higher-genus surfaces.\\

\begin{definition}\label{hom bracket}
	Define $\bar{c}(S)$ to be zero if $r(S) \neq 1$. If $r(S)=1$, then define $\bar{c}(S)$ to be one-half the number of homologically essential loops in $S$. The more general homological Kauffman bracket which will be used throughout this paper is given by
	\begin{equation}
	\langle D \rangle_\Sigma = \sum\limits_{S\in \mathcal{S}} A^{\alpha(S)-\beta(S)}(-A^2-A^{-2})^{k(S)} Z^{r(S)}W^{\bar{c}(S)},
	\end{equation}
	and we similarly write
	\begin{equation}
	\langle D \mid S \rangle_\Sigma = A^{\alpha(S)-\beta(S)}(-A^2-A^{-2})^{k(S)} Z^{r(S)} W^{\bar{c}(S)}
	\end{equation}
	for the contribution of a state.
\end{definition}

The value of a trivial loop is unaffected by the addition of the $W$ variable, so one can verify that Lemmas 2.1-2.3 in \cite{b-k} hold for this more general polynomial as well. From here, one can define a link invariant by making the same substitution and renormalization as in \eqref{hom jones}. This does not affect the coefficients, however, so $\langle D \rangle_\Sigma$ will suffice for our purposes. In Proposition 1.7 of \cite{b-k}, the authors prove that alternating, cellularly embedded, links are checkerboard colorable and that any checkerboard colorable link $L$ has $[L] = 0$ in $H_1(\Sigma \times [-1,1]; \mathbb{Z}/2)$. Under the projection, $L$ is $\mathbb{Z}_2$-homologous to the sum of all loops in any fixed state $S$. Thus, $[S]=0$ in $H_1(\Sigma; \mathbb{Z}/2)$ as well. As a result, if $r(S)=1$, then there must be an even number of essential loops in $S$, so we conclude that $\langle D \rangle_\Sigma \in \mathbb{Z}[A^{\pm 1},Z,W]$.\\

After having expressed certain coefficients of $\langle D \rangle_\Sigma$ in terms of the state graphs, we will form a connection between the coefficients and the twist number of D.\\

\begin{definition}\label{twist number}
	Let $D$ be a reduced link diagram on a surface $\Sigma$. By definition, the complementary regions of $D$ are $n$-gons with $n\geq 2$. A \textit{twist region} is a connected sequence of bigons arranged crossing-to-crossing of maximal length, as in Figure 3. The \textit{twist number} $tw(D)$ denotes the number of twist regions of $D$.
\end{definition}

\begin{figure}
	\centering
	\includegraphics{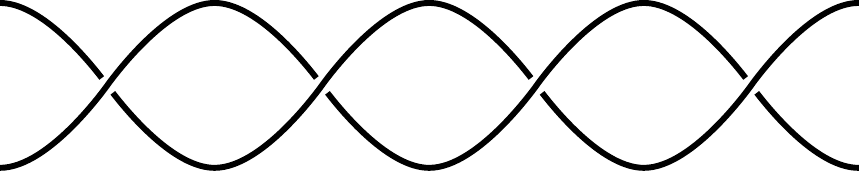}
	\caption{A twist region.}
\end{figure}

%\begin{definition}
%	Let $D$ be a cellularly embedded diagram on a surface $\Sigma$. Let $E$ be a disk on $\Sigma$ such that $\delta E$ intersects $D$ only in neighborhoods of two crossings as in Figure (?). We say that $D$ is twist-reduced if for any such $E$ either $E$ contains a string of bigons, or there is a disk in $\Sigma \setminus E$ opposite $E$ which contains a string of bigons.
%\end{definition}

\section{Computation of coefficients}\label{comp}

We first discuss a few properties of reduced alternating diagrams which result in some nice structure on $\Sigma$. Let $S_A$ denote the all-$A$ state, and let $S_B$ denote the all-$B$ state, each obtained by selecting only that type of resolution. If $D$ is alternating, and $l$ is any loop in either $S_A$ or $S_B$, then all surgery arcs attached to $l$ must lie on the same side of $l$ as seen in Figure 4. For $S\in\mathcal{S}$, let $\mathcal{A}(S)$ denote the collection of surgery arcs for $S$. Consider $U = S \cup \mathcal{A}(S)$. Observe that we can recover the underlying 4-valent graph $\Gamma$ by contracting each arc of $U$ to a point. This contraction induces a homeomorphism between the complementary regions of $U$ and the complementary regions of $\Gamma$. Therefore, if $D$ is cellularly embedded, the regions of $U$ must be disks. As a result, the side of $l$ containing no edges must be one of these disk regions. In other words, every loop of $S_A$ and $S_B$ is contractible. Therefore, we can use the states $S_A$ and $S_B$ to define the following graphs on $\Sigma$.\\

\begin{definition}\label{states + graphs}
	The \textit{all-A state graph} $\mathbb{G}_A$ is a graph embedded on $\Sigma$ whose vertices correspond to the loops of $S_A$ and whose edges correspond to the surgery arcs seen in Figure 1. The \textit{all-B state graph} $\mathbb{G}_B$ is defined similarly for $S_B$.
\end{definition}

Let $v_A$ and $v_B$ denote the number of vertices in $\mathbb{G}_A$ and $\mathbb{G}_B$ (alternatively, the number of loops in $S_A$ and $S_B$), respectively. Since all the loops in $S_A$ and $S_B$ are contractible, they are homologically trivial. This shows the states $S_A$ and $S_B$ contribute
\begin{figure}
	\centering
	\includegraphics{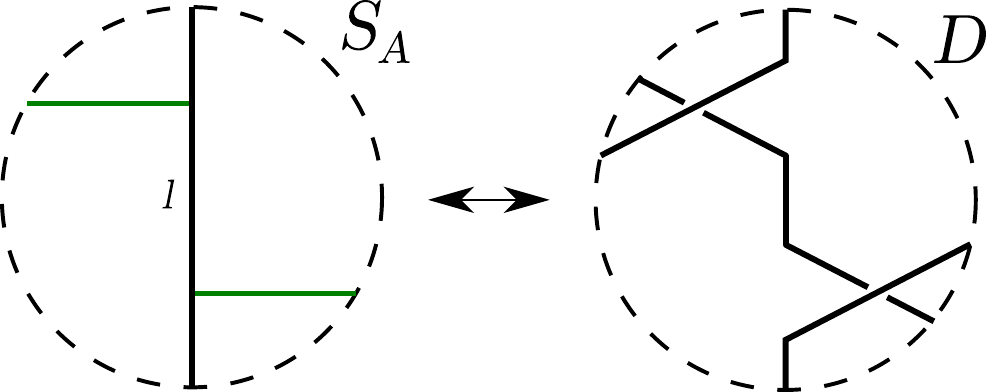}
	\caption{Arcs on opposite sides of $l$ contradict the fact that $D$ is alternating.}
\end{figure}
\begin{align}
	\langle D \mid S_A \rangle_\Sigma &= A^{c(D)}(-A^2-A^{-2})^{v_A}\\
	\langle D \mid S_B \rangle_\Sigma &= A^{-c(D)}(-A^2-A^{-2})^{v_B}
\end{align}
to $\langle D \rangle_\Sigma$. For a polynomial $p \in \mathbb{Z}[A^{\pm 1},Z,W]$, we write $\deg_{\max}^A(p)$ and $\deg_{\min}^A(p)$ for the maximal and minimal degrees of $p$ in the variable $A$, respectively. Note from the above, that
\begin{align}
	\deg_{\max}^A\langle D \mid S_A \rangle_\Sigma &= c(D) + 2v_A \\
	\deg_{\min}^A\langle D \mid S_B \rangle_\Sigma &= -c(D) - 2v_B.
\end{align}
In fact, these degrees are precisely the maximal and minimal degrees in $A$ of the polynomial $\langle D \rangle_\Sigma$. The authors of \cite{b-k} show this for $\langle D \rangle_\Sigma^H$ in their Proposition 2.8 by proving that reduced alternating diagrams on surfaces are homologically adequate.

\begin{definition}\label{adequate}
	A diagram $D$ is \textit{homologically $A$-adequate} if, for any state $S$ with exactly one $B$-resolution, we have $k(S)\leq k(S_A)$. A diagram $D$ is \textit{homologically $B$-adequate} if, for any state $S$ with exactly one $A$-resolution, we have $k(S)\leq k(S_B)$. A diagram $D$ is called \textit{homologically adequate} if it is both homologically $A$-adequate and homologically $B$-adequate.
\end{definition}

Lemma 2.6 of \cite{b-k} shows that for homologically adequate diagrams, $S_A$ and $S_B$ are the unique states which contribute maximal and minimal degree terms. This can be restated in terms of coefficients as the following.

\begin{proposition}\label{first coeff}(Lemma 2.6 in \cite{b-k})
	Let $D$ be a reduced alternating projection of a link $L$ onto a closed, orientable surface $\Sigma$. Then $\langle D \rangle_\Sigma^H$ has unique terms of maximal and minimal degree in the variable $A$ of the form
	\begin{equation}
		(-1)^{v_A} A^{c(D)+2v_A}
	\end{equation}
	and
	\begin{equation}
		(-1)^{v_B} A^{-c(D)-2v_B}.
	\end{equation}
\end{proposition}

Clearly, the addition of the variable $W$ does not affect the degrees in $A$ or $Z$, so the same holds for $\langle D \rangle_\Sigma$. This proposition directly corresponds to the classical result of Kauffman, Thistlethwaite, and Murasugi (\cite{kauffman}, \cite{thistlethwaite}, \cite{murasugi}) which led to a proof of the Tait conjectures using the Jones polynomial.\\

In the classical planar case, Dasbach and Lin in \cite{d-l} as well as Stoimenow in \cite{stoimenow} compute more coefficients. These coefficients have more complicated expressions in terms of data coming from the graphs $\mathbb{G}_A$ and $\mathbb{G}_B$. We will extend their work to links in surfaces using the more general homological Kauffman bracket. First, we set some terminology. Typically, the term ``loop" is used for an edge in a graph which is incident to only one vertex. Since we will be using the word ``loop" to refer to the simple closed curves in a state, we will instead call such an edge a \textit{self-edge}. An edge which is incident to two distinct vertices is called a \textit{simple edge}.

\begin{definition}
	Let $e$ be a self-edge in either $\mathbb{G}_A$ or $\mathbb{G}_B$. $e$ may viewed as a well-defined element $[e]$ of $H_1(\Sigma)$ up to a choice of orientation. We define an equivalence relation $\sim^*$ on self-edges, where $e_1 \sim^* e_2$ if $[e_1]=\pm[e_2]$ and if they are adjacent, meaning incident to the same vertex.\\
	
	Similarly, let $e_1$ and $e_2$ be simple edges in either $\mathbb{G}_A$ or $\mathbb{G}_B$ which are incident to the same pair of vertices. $e_1$ and $e_2$ form a cycle in the graph which may also be viewed as an element of $H_1(\Sigma)$ up to orientation. We define an equivalence relation $\sim$ on simple edges, where $e_1 \sim e_2$ if they form a null-homologous cycle.\\
	
	Let $e_A^*$ and $e_B^*$ denote the number of distinct equivalence classes of self-edges in $\mathbb{G}_A$ and $\mathbb{G}_B$, and let $\tilde{e}_A$ and $\tilde{e}_B$ be the number of distinct equivalence classes of simple edges in $\mathbb{G}_A$ and $\mathbb{G}_B$, respectively.
\end{definition}

\begin{proposition}\label{second coeff}
	Let $D$ be a reduced alternating projection of a link $L$ onto a closed, orientable surface $\Sigma$. Then, the terms of $\langle D \rangle_\Sigma$ having the second-highest degree in the $A$ variable are of the form
	\begin{equation}\label{second degree}
	(-1)^{v_A}\ (\alpha_{(1)}'W + \alpha_{(2)}'W^2 + \cdots +  \alpha_{(N)}'W^N)\ Z \ A^{c(D)+2v_A-2},
	\end{equation}
	where
	\begin{equation}\label{alpha prime}
		\alpha_{(1)}' = e_A^*
	\end{equation}
	and the terms having the second-lowest degree are of the form
	\begin{equation}
		(-1)^{v_B}\ (\beta_{(1)}'W + \beta_{(2)}'W^2 + \cdots + \beta_{(M)}'W^M )\ Z \ A^{-c(D)-2v_B +2},
	\end{equation}
	where
	\begin{equation}
	\beta_{(1)}'= e_B^*
	\end{equation}
	for some $N$, $M \geq 1$.
\end{proposition}

\begin{proof}
	We look first at states which contribute to the second-highest degree. Let $S\in\mathcal{S}$ be any state with at least one $B$-resolution. $S$ contributes
	\begin{equation}
	\langle D \mid S \rangle_\Sigma = A^{c(D) - 2\beta(S)}(-A^2-A^{-2})^{k(S)} Z^{r(S)} W^{\bar{c}(S)}
	\end{equation}
	with
	\begin{equation}\label{contribution}
		\deg_{\max}^A(\langle D \mid S \rangle_\Sigma) = c(D) - 2\beta(S) + 2k(S)
	\end{equation}
	so we need only consider states such that $k(S) - \beta(S)$ is maximal.\\
	
	Suppose $\beta(S) = n \geq1$ and let $e_1, \cdots, e_n$ be the edges of $\mathbb{G}_A$ that correspond to $B$-resolved crossings in $S$, ordered arbitrarily. Surgering along the edges one at a time produces a sequence of states $S_A = S_0, S_1, \cdots, S_n = S$. By homological adequacy, $k(S_1)\leq k(S_0)$ but for any other surgery, $k(S_{i+1}) \leq k(S_i)+1$. $k(S)$ will be maximal if these are all equalities, in which case $k(S)-\beta(S) = v_A + n-1$ and the corresponding degree in the $A$ variable will be $c(D) + 2v_A -2$. We will see that such states have homological rank one, thus proving \eqref{second degree}.\\
	
	We should note here, that in the binomial expansion of $(-A^2-A^{-2})^{v_A}$ in $\langle D \mid S_A \rangle_\Sigma$ all degrees of the variable $A$ will be of the form $c(D) +2v_A -4k$ for some $k$, so $S_A$ will not contribute to these coefficients. We now turn our attention to $\alpha_{(1)}'$, so assume further that $\bar{c}(S)=1$. Our goal is to understand what types of surgeries the edges $e_1, \cdots, e_n$ correspond to. A surgery from $S_i$ to $S_{i+1}$ is one of three types:
	
	\begin{enumerate}
		\item a \textit{merge}, where the total number of loops is reduced- in this case either $k(S_{i+1}) = k(S_i)-1$ and $r(S_{i+1}) = r(S_{i})$, or $k(S_{i+1})=k(S_{i})$ and $r(S_{i+1}) = r(S_{i})-1$
		
		\item a \textit{split}, where the total number of loops is increased- in this case either $k(S_{i+1}) = k(S_i)+1$ and $r(S_{i+1}) = r(S_{i})$, or $k(S_{i+1})=k(S_{i})$ and $r(S_{i+1}) = r(S_{i})+1$
		
		\item a \textit{single cycle smoothing} (defined in Section 2 of \cite{b-k}), where the total number of loops remains the same- in this case $k(S_{i+1}) = k(S_i)$ and $r(S_{i+1}) = r(S_{i})$
	\end{enumerate}

	  The key difference between the surface case and the planar case in \cite{d-l} and \cite{stoimenow} is the possibility that the kernel is preserved between states, as well as the potential for single cycle smoothings. However, as the authors of \cite{b-k} point out, single cycle smoothings cannot occur for checkerboard colorable diagrams. Also since every loop in $S_A = S_0$ is trivial, any merge would reduce the kernel. The only remaining possibility for the first surgery is a split with $k(S_1)=k(S_0)$ and $r(S_{1}) = r(S_{0})+1 = 1$. This means that $e_1$ is a self-edge for a loop $l$ and that surgering $l$ along $e_1$ produces parallel loops $l_1$ and $l_2$ in $S_1$ which are both homologous to $e_1$. Thus, $\bar{c}(S_1)=1$. Each remaining surgery must be a split where $k(S_{i+1}) = k(S_i)+1$, $r(S_{i+1}) = r(S_{i})$, and $\bar{c}(S_{i+1})=\bar{c}(S_i)$.\\
	  
	  Next, we will show by induction that the remaining edges are in the same equivalence class as $e_1$. Suppose that for all $j\leq i$ that $e_j$ is adjacent to $e_1$ and that $i_*(H_1(S_j))=i_*(H_1(S_1))=\langle e_1 \rangle \subset H_1(\Sigma)$. Let $l'$ be the loop in $S_i$ corresponding to the vertex incident to $e_{i+1}$. There are two possibilities.\\
	  
	  \textbf{Case 1} Suppose $l'$ is null-homologous. Because $r(S_{i+1}) = r(S_{i}) = 1$ and $\bar{c}(S_{i+1})=\bar{c}(S_i)$, surgering $l'$ along $e_{i+1}$ must produce two homologically trivial loops, as opposed to two parallel essential loops. Note that by homological adequacy, $l'$ must have been created by a previous surgery, so $e_{i+1}$ must have been adjacent to $e_j$ for some $j\leq i$. Therefore,  $i_*(H_1(S_{i+1}))=\langle e_1 \rangle \subset H_1(\Sigma)$ and the inductive step holds.\\
	  
	  \textbf{Case 2} Suppose $[l'] = [e_1]$. Since $\bar{c}(S_{i}) = \bar{c}(S_{1})=1$, $S_i$ has two essential loops, one of which is $l'$. Thus, $e_{i+1}$ must have been adjacent to $e_j$ for some $j\leq i$. Furthermore, the other essential loop homologous to $e_1$ is unaffected by the surgery. Since $r(S_{i+1}) = r(S_{i})=1$, surgering $l'$ along $e_{i+1}$ must produce one null-homologous loop and one loop homologous to $e_1$. This gives $i_*(H_1(S_{i+1}))=i_*(H_1(S_{i})) =\langle e_1 \rangle $ and thus, the inductive step holds.\\
	  
	  Therefore, $i_*(H_1(S)) = \langle e_1 \rangle \subset H_1(\Sigma)$ and all the edges are adjacent to $e_1$. Changing the order of  edge surgeries so that $e_i$ is done first will still yield the same state $S$, implying $[e_i]=[e_1]$ for all $i$. This means precisely that $e_i \sim^* e_1$ for all $i$.\\
	 
	  Let $S$ be a state whose $B$-resolutions occur within a single equivalence class of self-edges. Let $e_i^A$ be the $i$th of these classes, and let $k_i$ be the number of self-edges in that class. The total contribution of these states is
	  \begin{align*}
	  	&\sum\limits_{i=1}^{e_A^*} \sum\limits_{j=1}^{k_i} {{k_i} \choose j}A^{c(D)-2j}(-A^2-A^{-2})^{v_A + j-1}ZW\\
	  	= &\sum\limits_{i=1}^{e_A^*} \sum\limits_{j=1}^{k_i} {{k_i} \choose j}A^{c(D)-2j}(-1)^{v_A+ j-1}(A^2+A^{-2})^{v_A + j-1}ZW\\
	  	= &\sum\limits_{i=1}^{e_A^*} \sum\limits_{j=1}^{k_i} {{k_i} \choose j}A^{c(D)-2j}(-1)^{v_A + j-1}(A^{2v_A + 2j-2} + \text{l.o.t.})ZW\\
	  	= &\sum\limits_{i=1}^{e_A^*} \sum\limits_{j=1}^{k_i} {{k_i} \choose j}(-1)^{v_A + j-1}A^{c(D) + 2v_A -2}ZW + \text{l.o.t.}\\
	  	= &\sum\limits_{i=1}^{e_A^*} (-1)^{v_A}A^{c(D) + 2 v_A-2}ZW + \text{l.o.t.}\\
	  	= &(-1)^{v_A}e_A^*A^{c(D) + 2 v_A-2}ZW + \text{l.o.t.}
	  \end{align*}
	  proving \eqref{alpha prime}. The second-to-last equality follows from the fact that $\sum\limits_{j=0}^k (-1)^j {k \choose j} = 0$. An analogous argument holds for states whose $A$-resolutions lie in the same equivalence class of self edges and the coefficient $\beta_{(1)}'$.\\
	  
	  In the calculation above, note that in order for $\bar{c}(S)$ to equal one, the self-edges in $S$ must be adjacent. In general, $\bar{c}(S)$ counts the number of vertices which house edges that are homologous, but not equivalent under $\sim^*$. This is illustrated in Figure 5. The numbers $N$ and $M$ in the statement of the proposition are the maximum numbers of these vertices, taken over the set of all homology classes of self-edges.
	  
	   \begin{figure}
	  	\centering 
	  	\includegraphics[width=3in]{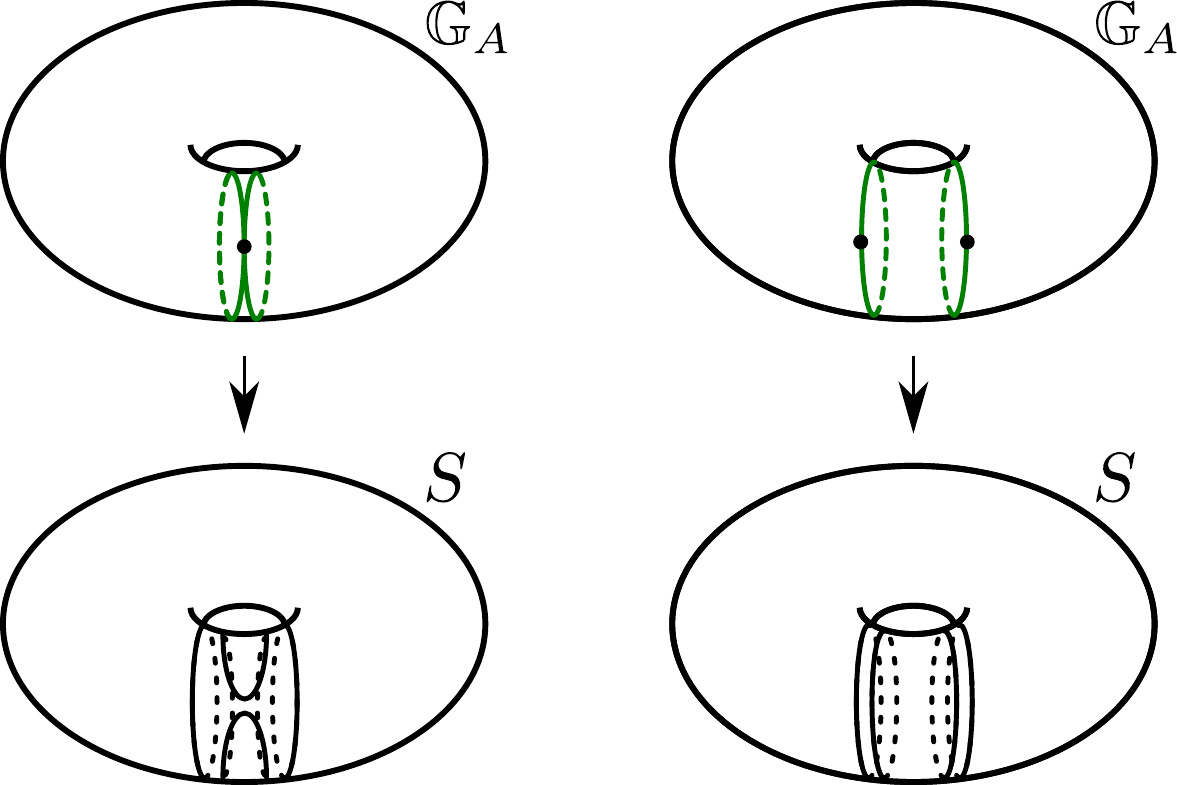}
	  	\caption{Left: Surgering along self-edges in the same equivalence class. Right: Surgering along merely homologous self-edges.}
	  \end{figure}
\end{proof}

Since it is possible for a diagram to have no self-edges, we cannot establish a meaningful lower bound for any combination of $\alpha_{(1)}'$ and $\beta_{(1)}'$. It is then necessary to look at more coefficients.

\begin{definition}
	Let $e_1$ and $e_2$ be adjacent self-edges in either $\mathbb{G}_A$ or $\mathbb{G}_B$ whose homology classes have intersection number $int([e_1],[e_2])=\pm 1$. We call the unordered double $(e_1,e_2)$ a \textit{transverse pair} of self-edges.\\
	
	If $(e_1,e_2)$ is a transverse pair, it is possible that there exists a third self-edge $e_3$ such that $[e_3]=\pm[e_1]\pm[e_2]$. In this case, we call the unordered triple $(e_1,e_2,e_3)$ a \textit{self-triangle}.\\
	
	$\sim^*$ induces equivalence relations on transverse pairs and self-triangles. Let $\pitchfork_A^*$ and $\pitchfork_B^*$ denote the number of equivalence classes of transverse pairs of self-edges, and let $\tau_A^*$ and $\tau_B^*$ denote the number of equivalence classes of self-triangles in $\mathbb{G}_A$ and $\mathbb{G}_B$, respectively. See Figure 6 for an example of a transverse pair and a self-triangle.
\end{definition}

\begin{proposition}\label{third coeff}
	Let $D$ be a reduced alternating projection of a link $L$ onto a closed, orientable surface $\Sigma$. Then, the terms of $\langle D \rangle_\Sigma$ having the third-highest degree in the variable $A$ are of the form
	\begin{equation}
		(-1)^{v_A}(\alpha''_{(0)} + \alpha''_{(2)}Z^2)A^{c(D)+2v_A -4},
	\end{equation}
	where
	\begin{equation}
		\alpha''_{(0)} = v_A -\tilde{e}_A + \pitchfork_A^*-\tau_A^* \text{ and } 
		\alpha''_{(2)} = {e_A^* \choose 2} - \pitchfork_A^*,
	\end{equation}
	and the terms having the third-lowest degree in the variable $A$ are of the form
	\begin{equation}
	(-1)^{v_B}(\beta''_{(0)} + \beta''_{(2)}Z^2)A^{-c(D)-2v_B +4},
	\end{equation}
	where
	\begin{equation}
	\beta''_{(0)} = v_B -\tilde{e}_B + \pitchfork_B^*-\tau_B^* \text{ and } 
	\beta''_{(2)} = {e_B^* \choose 2} - \pitchfork_B^*.
	\end{equation}
\end{proposition}

\begin{figure}
	\centering
	\includegraphics{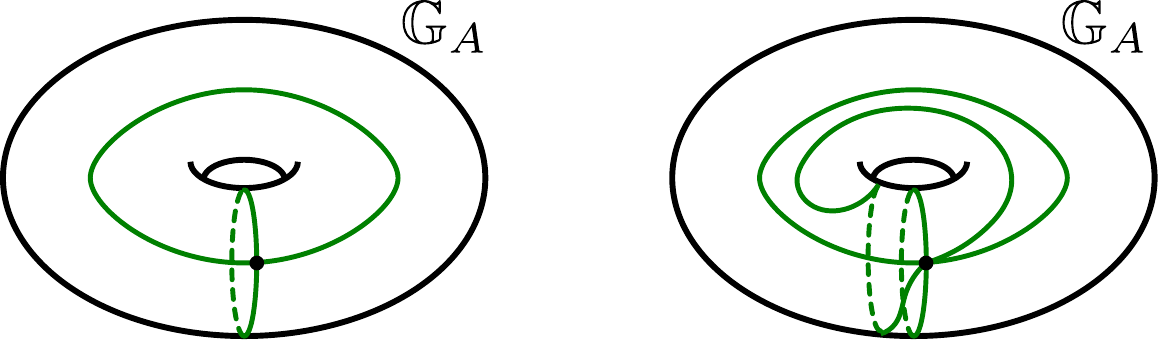}
	\caption{Left: A transverse pair. Right: A self-triangle.}
\end{figure}

Note that these coefficients contain elements of both second and third coefficients from the planar case in \cite{d-l} and \cite{stoimenow}. The reason for this is the possibility for the kernel to be preserved at the first surgery, whereas in the planar case the number of loops can only decrease at the first surgery. For the first surgery, only self-edges preserve the kernel, which is why the terms for self-triangles and pairs of self-edges appear here, while corresponding terms for simple edges do not.

\begin{proof}
	By \eqref{contribution}, the third-highest possible degree of $\langle D\rangle_\Sigma$ in the variable $A$ is $c(D)+2v_A -4$. Notice that this time, $S_A$ will make a contribution. We get
	\begin{align*}
		\langle D \mid S_A \rangle_\Sigma &=
		A^{c(D)}(-A^2-A^{-2})^{v_A}\\
		&= (-1)^{v_A}A^{c(D)}(\text{h.o.t.} + v_A A^{2v_A -4} + \text{l.o.t.})\\
		&= \text{h.o.t.} + (-1)^{v_A}v_A A^{c(D) + 2v_A -4} + \text{l.o.t.}
	\end{align*}
	
	On the other hand, states of the form described in the previous proposition will not contribute anything (the degrees of $A$ in the binomial expansion are of the wrong form: $c(D) + 2v_A -2-4k$, for some $k$).\\
	
	 Now suppose that $S$ is a state with $\deg_{\max}(\langle D \mid S \rangle_\Sigma) = c(D)+2v_A -4$. As in the previous proof, let $e_1, \cdots, e_n$ be a sequence edges of $\mathbb{G}_A$, and $S_A = S_0, S_1, \cdots, S_n = S$ the sequence of states obtained by surgering along each edge in order. As before, we seek to determine which types of surgeries could have yielded $S$. By homological adequacy, $k(S_1)\leq k(S_0)$. There are two cases, depending on whether the first surgery reduces the kernel or preserves the kernel.\\
	
	\textbf{Case 1} If $k(S_1) = k(S_0)-1$, then $e_1$ is a simple edge, whose surgery merges two trivial loops $l_1$ and $l_2$ in $S_0$ into a single trivial loop $l'$. Then $r(S_1) = 0$. In order for the remaining surgeries to yield the desired degree of $A$, they must all be splits which increase the kernel. In particular, $e_2$ must split a trivial loop in $S_1$ into two null-homologous loops. By homological adequacy, this initial loop must be $l'$ and furthermore, in $\mathbb{G}_A$, $e_2$ must have been a simple edge connecting the same two vertices as $e_1$. It follows that $e_1$ and $e_2$ must have formed a null-homologous curve. This means precisely that $e_1 \sim e_2$. By reordering the sequence of edges we see that $e_i\sim e_1$ for all $2\leq i \leq n$. The total contribution of states whose $B$-resolutions occur within a single equivalence class of simple edges is
	\begin{align*}
		\sum\limits_{i=1}^{\tilde{e}_A} \sum\limits_{j=1}^{k_i} & {{k_i} \choose j}A^{c(D)-2j}(-A^2-A^{-2})^{v_A + j-2}\\
		= &(-1)^{v_A-1}\tilde{e}_A A^{c(D) + 2v_A -4} + \text{l.o.t.},
	\end{align*}
	where the equality follows from the same steps used in the proof of Proposition \ref{second coeff}.\\
	
	\textbf{Case 2} If $k(S_1) = k(S_0)$ then $e_1$ was a self-edge as in the previous proposition, and $r(S_1)  = 1$. In order for the state $S$ to contribute to the desired degree of $A$, there must be a $j, 1< j \leq n$ such that $k(S_j) = k(S_{j-1})$, and for all $i \neq 1,j$ we have $k(S_i) = k(S_{i-1}) + 1$ while the homological rank is preserved. Note that the $j$-th surgery could either be a merge resulting in $r(S_n) = r(S_j) = 0$ or a split resulting in $r(S_n) = r(S_j) = 2$. First, let us deal with the scenario that the $j$-th surgery is a merge.\\
	
	\begin{figure}
		\centering
		\includegraphics{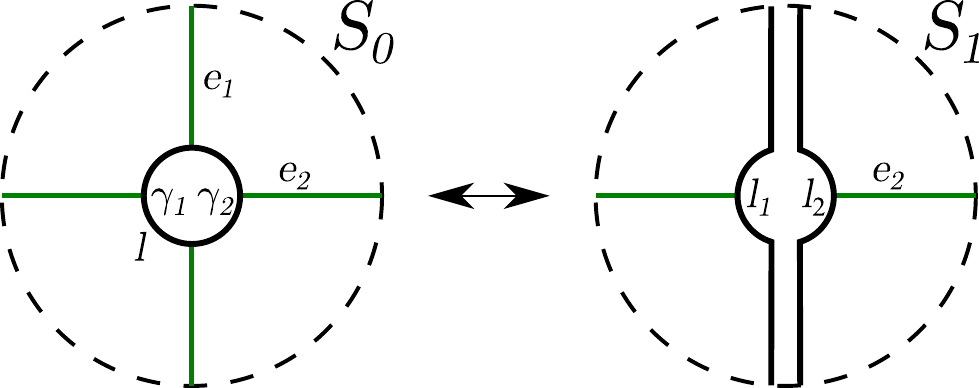}
		\caption{A state in case 2 (a).}
	\end{figure}
	
	\textbf{Case 2 (a)} We first claim that in such a scenario, all of the edges had to be adjacent self-edges. Indeed, proof of the previous proposition implies that $e_i$, $1\leq i<j$ are all homologous and that $H_1(S_i) = \langle e_1 \rangle$ (but we cannot yet conclude they are adjacent). However, since $r(S_j)$ must be zero, surgery along $e_j$ must be a merge which reduces the homological rank. This is only possible if $S_{j-1}$ contains exactly two essential loops $l_1$ and $l_2$, making $\bar{c}(S_{j-1})=1$. We saw that this can only happen if the $e_i$ are adjacent, and therefore belong to the same equivalence class for $1\leq i < j$. Surgery along $e_j$ merges $l_1$ and $l_2$ into a single null-homologous loop $l'$, so $e_j$ must have been adjacent to $e_i$ for some $i<j$. At this point, all loops in $S_j$ are null-homologous, so the remaining surgeries must split off more null-homologous loops. By homological adequacy, $e_i$, $j < i \leq n$ also must be adjacent to previous edges.\\
	
	Now that we know all edges in such a state must be incident to a single vertex $v$ in $\mathbb{G}$, it will suffice to look at a local neighborhood of the loop $l$ in $S_A$ corresponding to $v$. Without loss of generality, assume that $j=2$. As seen in Figure 7, $e_1$ splits $l$ into two arcs, $\gamma_1$ and $\gamma_2$ which correspond to the loops $l_1$ and $l_2$. Since surgery along $e_2$ merges $l_1$ and $l_2$, it must connect $\gamma_1$ and $\gamma_2$ outside of the disk region bounded by $l$. (This can only happen if $g(\Sigma)>0$). In this way, we see that $e_1$ and $e_2$ are transverse.\\
	
	Now, let us address the edges $e_i$, $2 < i \leq n$ whose surgeries increase the kernel. As seen in Figure 8, $e_1$ and $e_2$ divide $l$ into four arcs, $\gamma_k$, $k = \RN{1},\RN{2},\RN{3},\RN{4}$. $l' \in S_2$ is the boundary of a regular neighborhood of $l \cup e_1 \cup e_2$, which may be viewed as the union of corresponding arcs $\gamma_k'$ and two parallel copies of each of the edges $e_1$ and $e_2$. Consider $e_3$. By homological adequacy, the endpoints of $e_3$ must lie in different arcs. This leaves us with two possibilities.\\
	
	\begin{figure}
		\centering
		\includegraphics[width=3.5in]{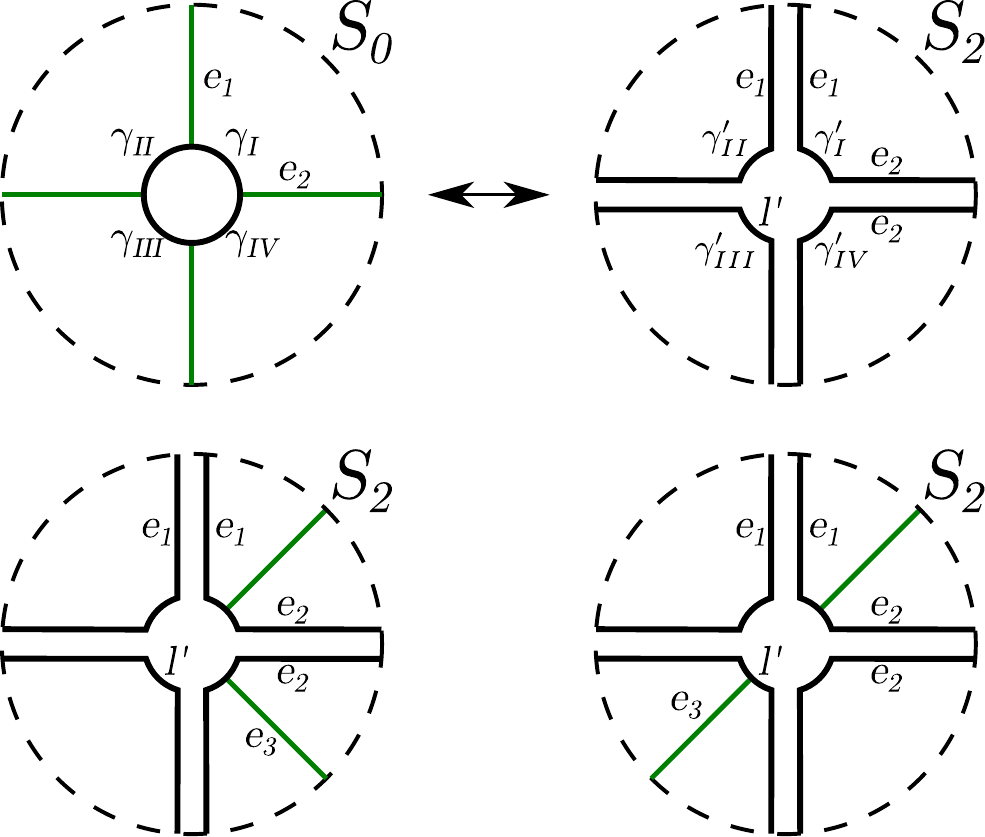}
		\caption{Bottom left: A state in case 2 (a-\Rn{1}). Bottom right: A state in case 2 (a-\Rn{2}).}
	\end{figure}
	
	\textbf{Case 2 (a-\Rn{1})} Suppose that $e_3$ connects two adjacent arcs. Since $k(S_3) = k(S_2) + 1$, surgery along $e_3$ splits off a null-homologous loop from $l'$. In other words, $e_3$ along with an arc in $l'$ forms a null-homologous loop, independent of the choice of arc. As seen in Figure 8, this arc may be taken to be one of the edges $e_i$, $i=1,2$. Thus, $e_3 \sim^*e_i$. If all remaining $e_i$, $3\leq i \leq n$ behave this way, then $S$ is a state whose $B$-resolutions appear within a single equivalence class of the transverse pair $(e_1,e_2)$, which we denote by $(e_1,e_2)^*$. The total contribution of such states is
	\begin{align*}
		\sum\limits_{(e_1,e_2)^*} \sum\limits_{i=1}^{k_{e_1}} \sum\limits_{j=1}^{k_{e_2}} & {{k_{e_1}} \choose i} {{k_{e_2}} \choose j} A^{c(D)-2(i+j)}(-A^2-A^{-2})^{v_A + i + j - 2}\\
		= &(-1)^{v_A}\pitchfork_A^*A^{c(D) + 2 v_A-4} + \text{l.o.t.},
	\end{align*}
	where $k_{e_1}$ and $k_{e_2}$ denote the number of edges in the equivalence classes of $e_1$ and $e_2$, respectively.\\
	
	\textbf{Case 2 (a-\Rn{2})} Suppose that $e_3$ connects two opposite arcs. This time, the arc in $l'$ is a combination of $e_1$ and $e_2$. This means precisely $[e_3] = \pm[e_1]\pm[e_2] \in H_1(\Sigma)$. Thus $(e_1, e_2, e_3)$ is a self-triangle. Note that if such an $e_3$ exists, no edge can connect the other pair of opposite $\gamma_k$ in $l$. Otherwise, if there did exist such an edge $e'$, consider the intersection number $int(e',e_3) = int(e',e_1) + int(e',e_2)$. Since $\mathbb{G}_A$ is embedded in $\Sigma$, we would have $int(e',e_1) = int(e',e_2) = int(e',e_1) + int(e',e_2) = \pm 1$ which is impossible. Therefore, for $3 < i \leq n$ $e_i$ falls under case 2 (a-\Rn{1}) or 2 (a-\Rn{2}) and as such belongs to the equivalence class of either $e_1$, $e_2$, or $e_3$. Then $S$ is a state whose $B$-resolutions appear within a single equivalence class of the self-triangle $(e_1,e_2,e_3)$, which we denote by $(e_1,e_2,e_3)^*$. The total contribution of such states is
	\begin{align*}
	\sum\limits_{(e_1, e_2, e_3)^*} \sum\limits_{i=1}^{k_{e_1}} \sum\limits_{j=1}^{k_{e_2}} \sum\limits_{k=1}^{k_{e_3}}&{{k_{e_1}} \choose i} {{k_{e_2}} \choose j} {{k_{e_3}} \choose k}A^{c(D)-2(i+j+k)}(-A^2-A^{-2})^{v_A + i + j + k - 2}\\
	= &(-1)^{v_A-1}\tau_A^*A^{c(D) + 2 v_A-4} + \text{l.o.t.}
	\end{align*}
	
	\textbf{Case 2 (b)} Lastly, we address the other scenario described at the beginning of Case 2, where the $j$th surgery is a split which increases the homological rank rather than the kernel. The only possibility left is that $e_j$ is a self-edge which neither transverse nor homologous to $e_1$. Once again, all other edges must belong to the  equivalence class of either $e_1$ or $e_j$. Letting $e_1^*$ and $e_2^*$ denote the two classes, the total contribution of these states is
	\begin{align*}
	\sum\limits_{\substack{e_1^* \neq e_2^* \\ \text{ not transverse}}} \sum\limits_{i=1}^{k_{e_1}} \sum\limits_{j=1}^{k_{e_2}} & {{k_{e_1}} \choose i} {{k_{e_2}} \choose j} A^{c(D)-2(i+j)}(-A^2-A^{-2})^{v_A + i + j  - 2}Z^2\\
	= &(-1)^{v_A}\Bigg({e_A^*\choose 2} - \pitchfork_A^*\Bigg)A^{c(D) + 2 v_A-4}Z^2 + \text{l.o.t.}
	\end{align*}
	
	Combining the contributions from $S_A$, Case 1, and Case 2 (a), we get $\alpha''_{(0)} = v_A -\tilde{e}_A + \pitchfork_A^*-\tau_A^*$ as claimed. Likewise, Case 2 (b) gives $\alpha''_{(2)} = {e_A^* \choose 2} - \pitchfork_A^*$. An analogous argument holds for $\beta_{(0)}''$ and $\beta_{(2)}''$.
\end{proof}

\section{Bounds on the twist number}\label{proofs}

The goal of this section is to prove Theorem \ref{main theorem}. Let $e_A = e_A^* + \tilde{e}_A$ and $e_B = e_B^* + \tilde{e}_B$. We consider the quantity
	\begin{align}\label{quant}
		\star &= \alpha_{(1)}' + \beta_{(1)}' - \alpha''_{(0)} - \beta''_{(0)} + 2\\
		&=  e_A + e_B - v_A - v_B + 2 + \tau_A^* + \tau_B^* - \pitchfork_A^* - \pitchfork_B^*.\nonumber
	\end{align}
	
We now seek to bound $\star$ in terms of the twist number of the diagram. We follow the method used for the planar case in \cite{dehn_filling}. As transverse edge pairs (and by extension self-triangles) do not arise in the planar case, we deal with these terms first.
	
	\begin{proposition}
		Suppose $\Sigma$ has genus $g$. Then,
		\begin{equation}\label{g bound}
			-2g \leq \tau_A^* + \tau_B^* - \pitchfork_A^* - \pitchfork_B^* \ \leq 0.
		\end{equation}
	\end{proposition}

	\begin{proof}
		The second inequality is immediate from the fact that every self-triangle consists of three self-edges which are all pairwise transverse. Furthermore, no pair can be part of any other self-triangle since the third self-edge is already determined by the pair.\\
		
		For the first inequality, we will show that $\tau_A^* + \tau_B^*\leq g$. Suppose without loss of generality that $\mathbb{G}_A$ contains a self-triangle class $(e_1, e_2, e_3)^*$ at a vertex corresponding to the loop $l$. Figure 9 shows the result of surgering $l$ along $e_1$, $e_2$, and $e_3$ to obtain the state $S$, where $e_1', e_2',$ and $e_3'$ are the surgery arcs dual to these edges. In place of the self-triangle, there is a region $R$ which is a punctured torus. Note that $\Sigma \setminus R$ then consists of two punctured surfaces with total genus $g-1$. Iterating this argument, we immediately see $\mathbb{G}_A$ can have at most $g$ self-triangle classes.\\
		
		\begin{figure}
			\centering
			\includegraphics{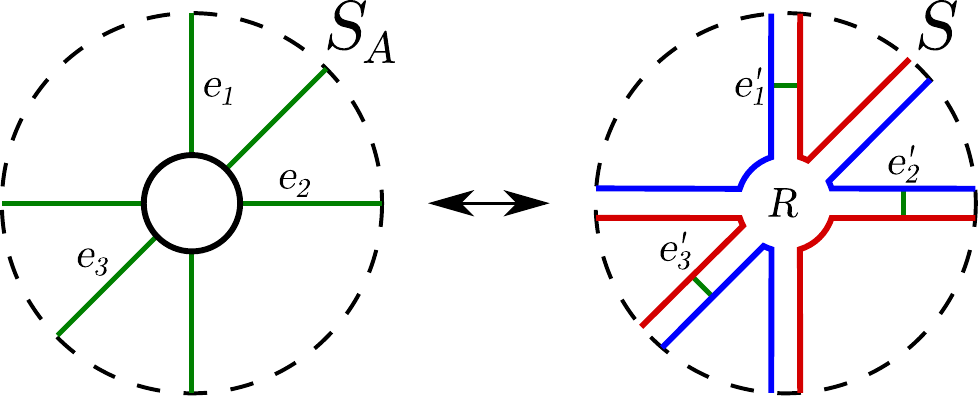}
			\caption{The result of surgering along a self-triangle. The red and blue loops bound separate punctured surfaces.}
		\end{figure}
		
		More is true, however. Notice that $e_1', e_2',$ and $e_3'$ are the only surgery arcs joining these punctured surfaces. Thus, when we complete the remaining surgeries to obtain $S_B$, $e_1', e_2',$ and $e_3'$ become identified with simple edges of $\mathbb{G}_B$. By reordering, this will also be the case for the other edges making up $(e_1, e_2, e_3)^*$. Therefore, the remaining surgeries induce an isotopy on $R$ taking it to a punctured torus in $\Sigma$ which contains no self-edges of $\mathbb{G}_B$. We have shown that the presence of a self-triangle class in either $\mathbb{G}_A$ or $\mathbb{G}_B$ obstructs the existence of a potential set of transverse self-edge classes (and by extension a self-triangle) in the other. Iterating this argument, we have $\tau_A^* + \tau_B^*\leq g$ as claimed.\\
		
		Recall that across both state graphs, there are a total of $3(\tau_A^* + \tau_B^*)$ classes of transverse pairs coming from self-triangles. By an analogous argument, in $\mathbb{G}_A$ there are at most $g-(\tau_A^*+\tau_B^*)$ pairs which are not part of any self-triangle. By contrast, these do not obstruct a transverse pair in $\mathbb{G}_B$: for an example, see the left-hand graph depicted in Figure 6 and note that it is isotopic to its dual.\\
		
		Therefore, the same bound holds for $\mathbb{G}_B$. By adding across both state graphs, we have
		\begin{align}
			\pitchfork_A^* + \pitchfork_B^* &\leq 3(\tau_A^* + \tau_B^*) + 2g - 2(\tau_A^* + \tau_B^*)\\
			&= 2g + \tau_A^* + \tau_B^*.\nonumber
		\end{align}
		
		This implies
		\begin{align}
			\tau_A^* + \tau_B^* - \pitchfork_A^* - \pitchfork_B^* &\geq -2g
		\end{align}
		as desired.
	\end{proof}

\begin{figure}[h]
	\centering
	\includegraphics[width=5in]{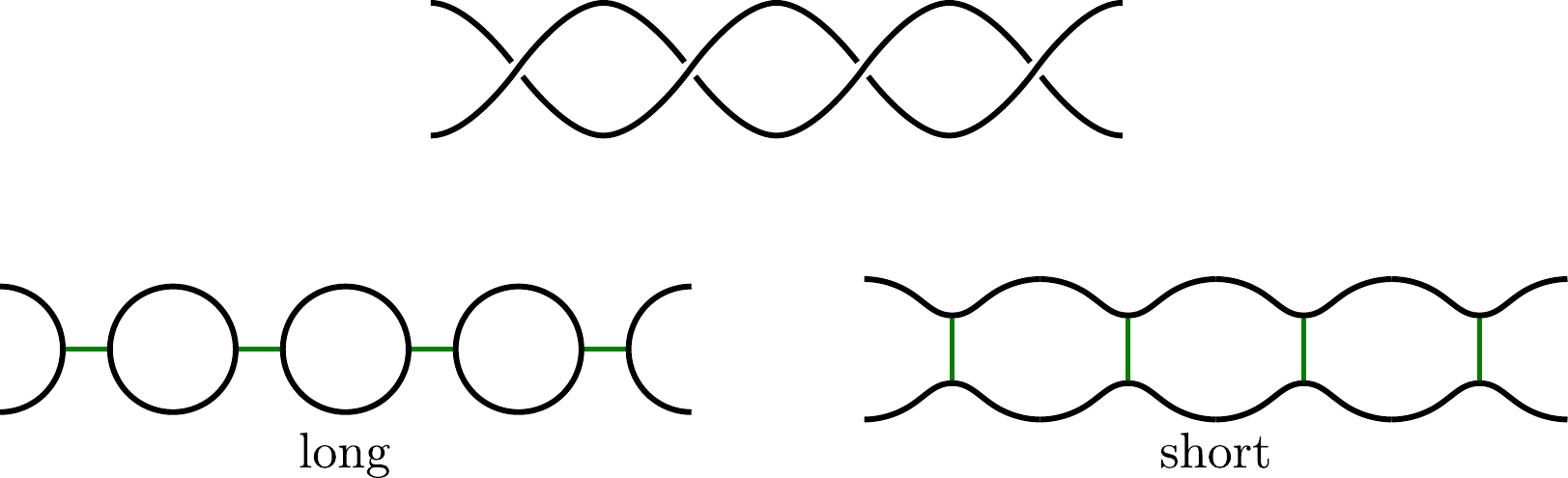}
	\caption{Left: The long resolution of a twist region. Right: The short resolution of a twist region.}
\end{figure}

In \cite{dehn_filling}, the authors bound the quantity analogous to $\star$ from above by recategorizing the vertices and edges based on how twist regions in $D$ reveal themselves in two different ways in each of the state graphs. Referring to Figure 10, we call these ``long" and ``short" resolutions. Edges are also called ``long" and ``short" if they come from that type of resolution. Notice that short edges coming from the same twist region are all in the same equivalence class. This is the exact same property exhibited by the ``reduced edges" in \cite{dehn_filling} which leads to Proposition 4.6, the upper bound in the planar case. As there is no difference in the definition of a twist region between the planar and higher-genus cases, the proof goes through without changes as long as one replaces ``reduced edges" with ``equivalence classes of edges."
	
	\begin{proposition}(Proposition 4.6 in \cite{dehn_filling})
	Let $D$ be a reduced alternating diagram on a closed, orientable surface $\Sigma$. Then,
		\begin{equation}\label{ub}
			e_A + e_B - v_A - v_B + 2 \leq 2 \ tw(D).
		\end{equation}
	\end{proposition}

For the lower bound, only a slight modification the methods used in the planar case is required. Let us first summarize the argument, which appears in section 4.3 of \cite{dehn_filling}. The authors first assume that there are no single-crossing twist regions, so that each region contains at least one bigon. Recall that $D$ may be viewed as a 4-valent graph $\Gamma$. They obtain a new 3-valent graph $P$ by collapsing strings of bigons in $\Gamma$ to red edges. A 2-valent graph, $\Phi$ may be obtained by deleting the red edges. Figure 11 depicts these graphs.\\

The authors call the regions of $P$ ``provinces" and the regions of $\Phi$ ``countries." In \cite{dehn_filling}, the graphs are all constructed on the Turaev surface of $D$, but since we only deal with alternating links, the graphs $P$ and $\Phi$ embed naturally on $\Sigma$. Note that $P$ inherits the cellular embedding of $D$, so all provinces are disks, while countries may be nontrivial regions. For each country, its provinces correspond to vertices in either $\mathbb{G}_A$ or $\mathbb{G}_B$ which are incident to more than two edges (the authors call these ``$n$-gon" vertices). The red edges which divide countries into provinces are dual to sets of parallel short edges which connect the two provinces. Recall that if short edges are dual to the same red edge, then they are in the same equivalence class. The converse is not true, however; there may exist short edges dual to distinct red edges which will become identified. The key to the authors' argument is to realize that these different short edges must lie in the same country. So, to find a lower bound the authors investigate the number of short edges in each country that will survive when we pass to equivalence classes. We now prove a version of their result (Lemma 4.8 in \cite{dehn_filling}).\\

\begin{figure}
	\centering
	\includegraphics[width=\textwidth]{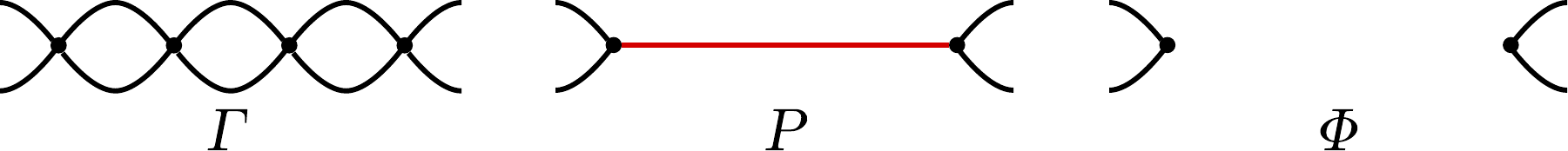}
	\caption{The graphs $\Gamma$, $P$, and $\Phi$.}
\end{figure}

\begin{proposition}\label{short_country}
	Let $D$ be a reduced alternating diagram on a closed, orientable surface $\Sigma$ such that every twist region of $D$ has at least two crossings. Let $N$ be a country of $D$, and let $e_{\text{short}}(N)$ denote the number of equivalence classes of short edges contained in $N$ and $e_{\partial}^*(N)$ the number of self-edge equivalence classes in $N$ which are homologous to a boundary component of $N$. Let $tw(N)$ denote the number of twist regions in $N$ and let $\lvert \partial N \rvert$ denote the number of loops in $\Phi$ which bound $N$. Then we have
	\begin{equation}\label{short_bound}
		e_{\text{short}}(N)\geq tw(N) + 1 - \lvert \partial N \rvert + e_{\partial}^*(N).
	\end{equation}
\end{proposition}

\begin{proof}
	Consider the dual graph of $N$, defined to be a connected ribbon graph $G$ whose vertices are the provinces of $N$, edges are dual to the red edges in $N$, and boundary components are the loops in $\Phi$ which bound $N$. In \cite{dehn_filling}, the authors obtain a lower bound by finding a spanning tree of $G$, or alternatively by cutting $N$ along red edges until a disk remains. In our case, it suffices to do the same with a spanning quasi-tree of $G$.\\
	
	A \textit{quasi-tree} is a ribbon graph with exactly one boundary component. We can obtain a spanning quasi-tree by removing edges as long as they reduce the number of boundary components of the graph, which corresponds to cutting $N$ along dual red edges. Thus, it suffices to cut $N$ along $\lvert \partial N \rvert - 1$ red edges.\\
	
	Note that if $G$ is a quasi-tree, then $N$ cannot contain any short edges which are in the same equivalence class, yet are dual to different red edges. Otherwise, these short edges form a null-homologous curve separating $N$, forcing there to be at least two boundary components. Also note that there can be no self-edges which are homologous to a boundary component of $N$ for the exact same reason. Recall that a cellularly embedded diagram must be connected, so that all self-edges are short.\\
	
	Thus, any classes we are left with correspond to a unique one of the remaining $tw(N) + 1 - \lvert \partial N \rvert$ red edges left in the country. Adding back in the self-edge classes which were removed, we obtain the desired inequality.
\end{proof}

The final step in \cite{dehn_filling} is to find a bound on the number of countries. A small but important modification is needed here. In the planar case, the authors show that every component of $\Phi$ contains at least three vertices. As we will see, this does not hold in the surface case, and this issue must be resolved in order to avoid canceling the term for the twist number in the lower bound.

\begin{proposition}\label{short_all}
	Let $D$ be a reduced alternating diagram on a closed, orientable surface $\Sigma$ such that every twist region of $D$ has at least two crossings. Let $e_{\text{short}}$ denote the total number of equivalence classes of short edges in $\mathbb{G}_A$ and $\mathbb{G}_B$. Then,
	\begin{equation}
		e_{\text{short}} \geq \frac{1}{3}tw(D) + 1 - g(\Sigma).
	\end{equation}
\end{proposition}

\begin{proof}
	Let $\phi$ be a component of $\Phi$ which may be thought of as a simple closed curve in $\Sigma$. Since $D$ is reduced, $\phi$ must contain at least two vertices, or else it would bound a monogon region corresponding to a nugatory crossing. We call $\phi$ ``bad" if it has exactly two vertices and ``good" if it has at least three. Let $\lvert \Phi \rvert_{\text{bad}}$ denote the number of bad curves and $\lvert \Phi \rvert_{\text{good}}$ the number of good curves, and $\lvert \Phi \rvert$ the total number of curves.\\
	
	Suppose $\phi$ is bad. Then because $P$ is cellularly embedded, the two red edges incident to the vertices of $\phi$ must lie on opposite sides of $\phi$. Otherwise, $\phi$ would bound a bigon region, and all bigons were collapsed in the construction of $P$. Also note that $\Phi$ inherits the checkerboard coloring of $D$, so we see that these are indeed distinct red edges which lie in separate countries. By following along $\phi$, we see that on each side of $\phi$ the red edge forms a border between a province and itself. Thus, the red edge is dual to a self-edge which represents the same homology class as $\phi$ in $H_1(\Sigma)$. This is illustrated in Figure 12. By homological adequacy, we have that this self-edge along with $\phi$ are homologically essential. In a country $N$, there can be at most two essential boundary components which are homologous. To see this, recall that these two homologous curves separate $\Sigma$, so a third would separate $N$. Therefore, for each bad curve there are two boundary-homologous classes of self-edges in the neighboring countries, and there is at worst a 2:1 correspondence between bad curves and a given self-edge class. Therefore, $e_\partial^* \geq \lvert \Phi \rvert_{\text{bad}}$ where $e_\partial^* = \sum\limits_N e_\partial^*(N)$.\\
	
	\begin{figure}
		\centering
		\includegraphics{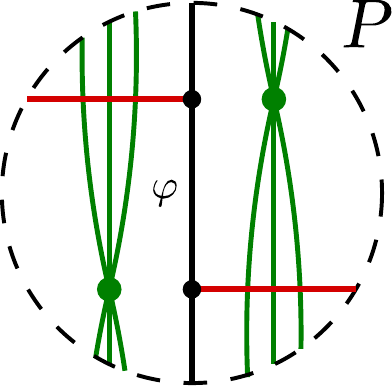}
		\caption{A bad loop. The boundary-homologous self-edges are drawn in green.}
	\end{figure}
	
	The total number of vertices in $P$ is at least three times the number of good edges. Since every red edge has two vertices, this gives us that $tw(D) \geq \frac{3}{2}\lvert \Phi \rvert_{\text{good}}$. Let $n(D)$ denote the number of countries in $\Phi$. By summing \eqref{short_bound} over all countries, we get
	\begin{align*}
		e_{\text{short}} &\geq tw(D) + n(D) - 2\lvert \Phi \rvert + e_{\partial}^*\\
		&\geq tw(D) + n(D) - 2\lvert \Phi \rvert + \lvert \Phi \rvert_{\text{bad}}\\
		&= tw(D) + n(D) - \lvert \Phi \rvert - \lvert \Phi \rvert_{\text{good}}\\
		&\geq \frac{1}{3}tw(D) + n(D) - \lvert \Phi \rvert\\
		&\geq \frac{1}{3}tw(D) + 1 - g(\Sigma).
	\end{align*}
	
	The last inequality follows from the fact that if we cut $\Sigma$ along the simple closed curves in $\Phi$, at most $g$ of these cuts are non-separating.
\end{proof}

We combine these results into a proposition mirroring Theorem 4.10 in \cite{dehn_filling}.\\

\begin{proposition}
	Let $D$ be a reduced alternating diagram on a closed, orientable surface $\Sigma$ such that every twist region of $D$ has at least three crossings. Then,
	\begin{equation}\label{lb}
		\star \geq \frac{1}{3}tw(D) + 1 - g(\Sigma)
	\end{equation}
\end{proposition}

\begin{proof}
	Like the authors in \cite{dehn_filling}, we define $v_{\text{bigon}}$ as the total number of bigon vertices in $\mathbb{G}_A$ and $\mathbb{G}_B$. Bigon vertices are incident to exactly two edges and correspond to the vertices in the long resolution of a twist region. We define $v_{\text{n-gon}}$ to be the total number of $n$-gon vertices. We have already seen that $e_{\text{short}}$ counts the total number of equivalence classes of short-edges, which are edges coming from a short resolution. Likewise, we define $e_{\text{long}}$ to be the total number of equivalence classes of edges coming from a long resolution. We remark that an edge is long if and only if it is incident to at least one bigon vertex. (When every twist region has at least three crossings, no two long edges can be in the same equivalence class.) Every vertex and edge across both state graphs falls into exactly one of these categories, so we can regroup $e_A + e_B - v_A - v_B = e_{\text{short}} + e_{\text{long}} -v_{\text{bigon}} - v_{\text{n-gon}}$.\\
	
	The long resolution of a twist region with $c$ crossings has exactly $c-1$ bigon vertices and $c$ long edges. By summing over every twist region, we see that $v_{\text{bigon}} = c(D)-t(D)$ and $e_{\text{long}} = c(D)$. Recall that the $n$-gon vertices become provinces in $P$, and the twist regions become red edges. Therefore, by summing over countries, we can compute $\chi(\Sigma) = v_{\text{n-gon}} - tw(D)$. Putting this all together,
	\begin{align*}
		\star &= e_A + e_B - v_A - v_B + 2 + \tau_A^* + \tau_B^* - \pitchfork_A^* - \pitchfork_B^*\\
		&= e_{\text{short}} + e_{\text{long}} -v_{\text{bigon}} - v_{\text{n-gon}} + 2 + \tau_A^* + \tau_B^* - \pitchfork_A^* - \pitchfork_B^*\\
		&= e_{\text{short}} + c(D) -(c(D) - tw(D)) -v_{\text{n-gon}}+ 2 + \tau_A^* + \tau_B^* - \pitchfork_A^* - \pitchfork_B^*\\
		&= e_{\text{short}} - \chi(\Sigma) + 2 + \tau_A^* + \tau_B^* - \pitchfork_A^* - \pitchfork_B^*\\
		&\geq e_{\text{short}} - \chi(\Sigma) + 2 -2g(\Sigma)\\
		&= e_{\text{short}}\\
		&\geq \frac{1}{3}tw(D) + 1 - g(\Sigma).
	\end{align*}
\end{proof}

Combining \eqref{ub} and \eqref{lb}, we obtain a proof of Theorem \ref{main theorem}.\\

In \cite{links_surfaces}, the authors bound the hyperbolic volume of the complement of a ``weakly generalized" alternating link in terms of the twist number. A weakly generalized alternating link is a link in a compact, irreducible, orientable 3-manifold with a projection onto a surface $\Sigma$ with some additional properties that guarantee the complement admits a complete hyperbolic structure. They also require the diagram to be twist-reduced. See \cite{links_surfaces} for the specific definitions of these terms.

\begin{theorem}(Theorem 1.4 in \cite{links_surfaces})
	Let $\Sigma$ be a closed orientable surface of genus at least one, and let $L$ be a link that admits a twist-reduced weakly generalized cellularly embedded alternating projection $D$ onto $\Sigma \times \{0\}$ in $Y=\Sigma \times [-1,1]$. Then the interior of $Y\setminus L$ admits a hyperbolic structure. If $\Sigma$ is a torus, then we have
	\begin{equation}
		\frac{v_{\text{oct}}}{2}\cdot tw(D) \leq vol(Y\setminus L) < 10v_{\text{tet}}\cdot tw(D),
	\end{equation}
	where $v_{\text{tet}}$ is the volume of a regular ideal tetrahedron, and $v_{\text{oct}}$ is the volume of a regular ideal octahedron.\\
	
	If $\Sigma$ has genus at least two,
	\begin{equation}
		\frac{v_{\text{oct}}}{2}\cdot (tw(D)-3\chi(\Sigma)) \leq vol(Y\setminus L) < 6v_{\text{oct}}\cdot tw(D).
	\end{equation}
\end{theorem}

Direct substitution of \eqref{bounds} into these volume bounds yields Corollary \ref{vol bound}.

\newpage

\printbibliography

\end{document}